\documentclass[12pt]{article}
\usepackage{amssymb,latexsym}
\usepackage{amsmath}
\usepackage{tikz}
\usepackage{verbatim}
\usetikzlibrary {positioning}
  \usetikzlibrary{arrows,topaths}
       \usetikzlibrary{decorations,backgrounds,snakes}
    \usetikzlibrary{shapes,snakes}
\usepackage{verbatim}
\usetikzlibrary{calc,shapes}   
\usetikzlibrary{decorations.pathreplacing,angles,quotes}
\usepackage{rawfonts}
\usepackage{amsmath, psfrag}
\usepackage{amsfonts}
\usepackage{amsthm}
\usepackage{lineno,hyperref}
\definecolor {processblue}{cmyk}{0.96,0,0,0}
\newtheorem{theorem}{Theorem}
\newtheorem{observation}[theorem]{Observation}

\newtheorem{problem}[theorem]{Problem}

\newtheorem{conjecture}[theorem]{Conjecture}
\newtheorem{proposition}[theorem]{Proposition}
\def\qed{\ifhmode\unskip\nobreak\fi\quad\ifmmode\Box\else$\Box$\fi}

\def\qed{\hfill $\Box$}

\author{Andr\'{e} E. K\'{e}zdy \\
	\small Department of Mathematics\\[-0.8ex]
	\small University of Louisville\\[-0.8ex]
	\small Louisville, Kentucky, U.S.A.\\[-0.8ex]
	\small\tt kezdy@louisville.edu\\
	\\
	Jen\H{o} Lehel \\
\small Department of Mathematics\\[-0.8ex]
	\small University of Louisville\\[-0.8ex]
	\small Louisville, Kentucky, U.S.A.\\
                \small{and}\\          
           \small  Alfr\'ed R\'enyi Institute of Mathematics\\[-0.8ex] 
            \small  Budapest, Hungary\\[-0.8ex]
              	\small\tt lehelj@renyi.hu}  
 
\begin{document}
\title{The equivalence of the Szemer\'edi and Petruska conjecture  
and the maximum order of \\$3$-uniform $\tau$-critical  hypergraphs}
\maketitle

\begin{abstract}
Recently we \cite{KL} asymptotically resolved the long-standing Szemer\'edi and Petruska conjecture \cite{SzP}.  
Several decades ago Gy\'arf\'as et al. \cite{GYLT} observed, via a 
straightforward but unpublished argument, that this conjecture is
equivalent to the problem of determining the maximum order of a $3$-uniform $\tau$-critical hypergraph.  
Consequently, an asymptotically tight upper bound for the maximum order of a $3$-uniform $\tau$-critical hypergraph follows 
from our recent work, reawakening interest in this equivalence.

In this companion paper we supply a simple proof of this equivalence.  We also present related background with open problems, and mention  combinatorial geometry applications of the Szemer\'edi and Petruska conjecture.
\end{abstract} 
\section{Introduction}
The two extremal problems indicated in the title are essentially the same through a straightforward complementary argument observed by Gy\'arf\'as et al. \cite{GYLT} decades ago.  The long-standing Szemer\'edi and Petruska conjecture \cite{SzP}  was recently resolved asymptotically by K\'ezdy and Lehel \cite{KL}.  As a corollary, an asymptotically tight upper bound follows for the maximum order of a $3$-uniform $\tau$-critical hypergraph. Weaker bounds were given earlier by using the theory of $\tau$-critical hypergraphs.  In contrast, K\'ezdy and Lehel \cite{KL} apply the iterative technique introduced by Szemer\'edi and Petruska; the iterative private pair technique there is reconsidered, substantially refined, then ultimately combined with the skew version of Bollob\'as theorem \cite{BB} on cross-intersecting set pair systems.  This strategy asymptotically resolves the Szemer\'edi and Petruska conjecture. 
The success of this approach has reawakened interest in the equivalence of the two problems mentioned in the title.  
This paper completes its companion \cite{KL} by supplying a simple proof of this equivalence.

In Section \ref{arrow} we begin by recalling the Hajnal-Folkman  lemma \cite{Hajnal} which can be considered the forefather of all the extremal problems considered here, including the very general family of `arrow symbol' problems introduced by Erd\H{o}s \cite{E}. A special case of these latter problems was investigated by Szemer\'edi and Petruska \cite{SzP} leading to their conjecture, see Section \ref{r=3} (Conjecture \ref{conj}). Sections \ref{comp} and \ref{taucritical} show that the two problems mentioned in the title are equivalent for general $r$-uniform hypergraphs (Proposition \ref{same}). It is important to emphasize that these ideas are not original. The observation that the two problems are equivalent goes back to the early work of Gy\'arf\'as et al. \cite{GYLT} and Tuza \cite{T}. The equivalence was exploited, for $r=3$, with a combinatorial geometry application by Jobson et al. \cite[Lemma 3]{JKL}. We conclude the note with open problems in Section \ref{quest}. In particular, we propose a question that generalizes the Hajnal-Folkman lemma, interpreting it as a problem on the maximum cliques of $r$-uniform hypergraphs (Problem \ref{HF}).
 
\section{The Hajnal-Folkman lemma}
\label{arrow}
  The Hajnal-Folkman  lemma states (see \cite{E,Hajnal}):
 {If a graph has  at most $2k-1$ vertices, where $k$ is the  maximum clique size, then
its  maximum cliques  share a common vertex.} Generalizing this lemma to set systems, Erd\H{o}s \cite{E}
introduced  an ``arrow notation'' for a class of extremal problems which we now describe. 

For $1\leq \ell$ and $3\leq r\leq k\leq n$, let  $\mathcal{K}=\{N_1,\dots, N_\ell\}$  be a family of sets containing at least $k$ elements and let  $\big|\bigcup\limits_{i=1}^\ell N_i\big|=n$. The family
$\mathcal{K}$ generates an $r$-uniform hypergraph $H$ on vertex set $V=\bigcup\limits_{i=1}^\ell N_i$ such that an
 $r$-element subset $R\subset V$  is an edge of $H$ if and only if  $R\subset N_i$, for some $N_i\in\mathcal{K}$. In particular, each $N_i\in\mathcal{K}$ becomes a complete $r$-uniform subhypergraph of $H$ called a `clique'. 
 Erd\H{o}s's arrow symbol $(n,k,t)^r\rightarrow u$ denotes the claim that  for every $\mathcal{K}$, if the sets of $\mathcal{K}$ have no $t$-element transversal (a $t$-set meeting each $N_i$), then the hypergraph generated by $\mathcal{K}$ contains a $u$-clique (a clique with $u$ vertices).\footnote{ To emphasize fixed $r$ we abbreviated  the syntax $(n,k,t,r)\rightarrow u$   introduced by  Erd\H{o}s \cite{E} to  $(n,k,t)^r\rightarrow u$.}
 
 Accordingly, the form $(n,k,t)^r\not\rightarrow u$
 means: there exists a family $\mathcal{K}=\{N_1,\dots, N_\ell\}$ as above having no $t$-transversal, and such that the 
 $r$-uniform hypergraph $H$ generated by $\mathcal{K}$ has $n$ vertices and contains no $u$-clique. 
 One can consider this hypergraph $H$ as a `witness' for  $(n,k,t)^r\not\rightarrow u$. 
 
Referring to the negative form of the arrow symbol for $u=k+1$ and $t=1$, an  $r$-uniform hypergraph
 $H$ of order $n$ is defined to be an $r$-uniform {\it $(n,k)$-witness hypergraph} (a witness for $(n,k,1)^r\not\rightarrow k+1$) provided its clique number $\omega(H)=k$ and the $k$-cliques of $H$ have no common vertex. 
 
The Hajnal-Folkman lemma states that $(2k-1,k,1)^2\rightarrow k+1$; consequently, a witness graph for $(n,k,1)^2\not\rightarrow k+1$ has $n\geq 2k$ vertices. Rewriting this in terms of $m=n-k$, the lemma says that the  order of an $(n,k)$-witness graph is at most $2m$, which bound is actually tight.
The Szemer\'edi and Petruska conjecture concerns the case $r=3$. It
 states that, in terms of $m=n-k$, the maximum order of a $3$-uniform  $(n,k)$-witness hypergraph is ${m+2\choose 2}$.
 
\section{Complementarity}
\label{comp}
Here all definitions assume a fixed positive integer $r \ge 2$.  Let $H=(V,E)$ be an $r$-uniform  hypergraph.
The {\it independence number} $\alpha(H)$  and the {\it transversal number} $\tau(H)$ of $H$  are defined as
\begin{eqnarray*}
\alpha(H)&=&max\{|S| : S\subset V, \hbox{ if } R\subset S, |R|=r, \hbox{ then } R\not\in E\}\\
 \tau(H)&=&min \{|T| : T\subset V,  \;e\cap T\neq \varnothing \hbox{ for each } e\in E\}.
 \end{eqnarray*}

A set $T\subset V$ such that   $e\cap T\neq \varnothing$  for each  $e\in E$  is called a {\it transversal} of $H$. 
A transversal $T$ such that $|T|=\tau(H)$ is a {\it minimum  transversal} of $H$.
The {\it clique number}  of $H$ is defined to be 
$$\omega(H)=max\{|N| : N\subset V, \hbox{ if } R\subset N, |R|=r, \hbox{ then } R\in E\}.$$
Notice that, by definition, if $T$ is a minimum transversal, then $S=V\setminus T$ is a largest independent set. Therefore, we obtain $\tau(H)=|V|-\alpha(H)$, a hypergraph extension of one of the graph identities
due to  Gallai \cite{G}.
 
 Define 
$\widehat{H}=(V,\widehat{E})$, where $\widehat{E}$ 
contains all $r$-element subsets of $V$ not in $E$. 
Obviously, a clique of $H$ is an independent set in $\widehat{H}$, furthermore,
$\widehat{\widehat{H}}=H$. Summarizing these complementarity properties we obtain:
\begin{observation}\label{gallai}
If $H$ is an $r$-uniform hypergraph of order $n$, then $$n=\alpha(H)+\tau(H)=\omega(\widehat{H})+\tau(H)=\omega(H)+\tau(\widehat{H}).$$
\end{observation}

\section{$\tau$-critical hypergraphs}
\label{taucritical}
A hypergraph $H=(V,E)$ is
{\it $\tau$-critical} if it has no isolated vertex ($\bigcup\limits_{e\in E} e= V)$ and  $\tau(H-e)=\tau(H)-1$, where
$H-e$ is the partial hypergraph with vertex set $V$ and edge set $E\setminus{e}$. 
\begin{observation}
\label{vertexinout}
In an $r$-uniform $\tau$-critical hypergraph ($r\geq 2$),  for every vertex $v$
there is a minimum transversal containing $v$, and
there is a minimum transversal not containing $v$.\qed
\end{observation} 

A hypergraph $H=(V,E)$ is {\it vertex critical} if every $v\in V$ belongs to some minimum transversal
of $H$. Notice that, due to Observation \ref{vertexinout}, for fixed transversal number $t$ the family of
$\tau$-critical hypergraphs is contained in the family of vertex critical hypergraphs. 
\begin{proposition} [Gy\'arf\'as et al. \cite{GYLT}]
\label{critical}
A hypergraph $H=(V,E)$ is vertex critical if and only if every $\tau$-critical partial hypergraph $H^\prime=(V^\prime, E^\prime)$ defined by $E^\prime\subset E$, 
$V^\prime=\cup \{e\mid e\in E^\prime\}$ with  $\tau(H^\prime)=\tau(H)$  satisfies $V^\prime=V$.\qed
\end{proposition}

 Szemer\'edi and Petruska \cite{SzP} proved a bound on the maximum order of $3$-uniform $(n,k)$-witness hypergraphs.
 Gy\'arf\'as et al. \cite{GYLT} investigated the maximum order, $v_{max}(r,t)$, of an $r$-uniform $\tau$-critical hypergraph $H$ with $\tau(H)=t$. It was observed in \cite{GYLT}
 that these two extremal problems are essentially the same, so both works address the same function $g(r,t)$ using different techniques. 
 Here is a proof of this equivalence.
\begin{proposition}
\label{same}
 $v_{max}(r,t)\leq g(r,t)$  for some function $g(r,t)$ if and only if   each $r$-uniform $(n,k)$-witness hypergraph with  $k=n-t$ satisfies  $n\leq g(r,t)$.
\end{proposition}
\begin{proof}
Let $H=(V,E)$ be an $r$-uniform $(n,k)$-witness hypergraph
with $|V|=n$, and $\omega(H)=k=n-t$. By Observation \ref{gallai}, 
$k=\omega(H)=\alpha(\widehat{H})=n-\tau(\widehat{H}),
$ hence $ \tau(\widehat{H})=n-k=t$.
Since the $k$-cliques of $H$
have no common vertex, every $x\in V$ belongs to the complement of some $k$-clique of $H$, that is to a minimum transversal  of $\widehat{H}$. Therefore $\widehat{H}$ is vertex critical.
 Due to Proposition \ref{critical}, $\widehat{H}$ has an $r$-uniform $\tau$-critical partial hypergraph spanning $V$. Thus
 $n=|V|\leq v_{max}(r,t)\leq g(r,t)$ follows.
 
For the converse, assume to the contrary that
$H$ is an $r$-uniform  $\tau$-critical hypergraph of order 
$n=v_{max}(r,t)>g(r,t)$.  
By Observation \ref{gallai},  $\tau(H)=t$ implies 
$\omega(\widehat{H})=n-t$.
By Observation \ref{vertexinout}, the minimum transversals of $H$ have no common vertex and their union covers $V$, therefore,
the union of the $k$-cliques of  $\widehat{H}$ covers $V$ and these $k$-cliques
have no common vertex.
In other words, $\widehat{H}$ is an $r$-uniform $(n,k)$-witness hypergraph, where $k=n-t$. Thus we obtain
$g(r,t)<v_{max}(r,t)=|V|=n\leq g(r,t)$, a contradiction.
\end{proof}

  \section{The case  $r=3$}
  \label{r=3}
Szemer\'edi and Petruska \cite{SzP} obtained
the estimation $n\leq 8t^2+ 3t$ for  the order of a $3$-uniform $(n,k)$-witness hypergraph, which is equivalent with 
 $v_{max}(3,t)\leq 8t^2+ 3t$, by Proposition \ref{same}. They gave a lower bound construction and made the tight conjecture
 that can be phrased as follows.
\begin{conjecture}[Szemer\'edi and Petruska \cite{SzP}]
\label{conj}
For $n=k+t$, if $H$ is a $3$-uniform $(n,k)$-witness hypergraph, then $n\leq  {t+2\choose 2}$.
\end{conjecture}

 The upper bound was improved  by Gy\'arf\'as et al.  \cite{GYLT}  to $v_{max}(3,t)\leq 2t^2+t$, and later by Tuza\footnote{ Zs. Tuza, personal communication ($2019$).} to $v_{max}(3,t) \leq \frac{3}{4}t^2+t+1$.  We \cite{KL} recently proved that the Szemer\'edi and Petruska conjecture is asymptotically correct, which, due to Proposition \ref{same}, immediately implies the asymptotically tight bound $v_{max}(3,t) = {t+2\choose 2}+ {\rm O}(t^{5/3})$. 
 
 The  Szemer\'edi and Petruska conjecture was verified for $t=2,3$ and $4$ by Jobson et al. \cite{small}. The resolution of   the conjecture has applications in combinatorial geometry problems related to a question posed by Petruska and another one by Eckhoff  involving convex sets in the plane, see
  \cite{JKLPT,JKL}.
  
 \section{Further problems}
\label{quest}
The early results due to Szemer\'edi and Petruska  \cite{SzP} motivated further research and lead to unexpected applications.
 In the introduction to their 1972 paper (that is titled part I.) they write,  `` ...[for larger $r$] we get a more general problem and we are to return to it in a forthcoming paper''.  As we learned\footnote{ Gy. Petruska, personal communication ($2018$).} a few years ago, they never revisited this work.  So the innovative iterative approach introduced in their paper was never extended for $r>3$;  the technique was forgotten for a while. Their conjecture survived a decade later in a different setting as a problem on $\tau$-critical hypergraphs. 

Gy\'arf\'as et al. \cite{GYLT} proved general bounds  pertaining to $r$-uniform hypergraphs: $$ {t+r-2\choose r-1}+(t+r-2)\leq v_{max}(r,t)\leq t^{r-1} +t {t+r-2\choose r-2}.$$
The question was asked there whether the correct value  for $t\geq r$ (or the asymptotic) of $v_{max}(r,t)$  is the lower bound above  (cf. Tuza \cite[Problem 18]{T}).

\begin{problem}
\label{anyr}  Is
 $  v_{max}(r,t)= {t+r-2\choose r-1}+(t+r-2)$ ($t\geq r$)
  true (or true asymptotically)?
\end{problem}

We conclude the note by returning to the Hajnal-Folkman lemma, and extend it from graphs to a problem on hypergraphs. 
The original form of the lemma according to Hajnal \cite{Hajnal} is as follows. If $k$ is the size of the maximum cliques in a graph on $n$ vertices, and 
$\{N_1,\dots, N_\ell\}$ is the family of its $k$-cliques, then $\big|\bigcap\limits_{i=1}^\ell N_i\big|\geq 2k-n$.

\begin{problem} 
\label{HF}
If $k$ is the size of the maximum cliques in an $r$-uniform hypergraph on $n=k+m$ vertices, and 
$\{N_1,\dots, N_\ell\}$ is the family of its $k$-cliques, then
$$
\left|\bigcap\limits_{i=1}^\ell N_i\right|  \geq n- \left[{m+r-2\choose r-1}+(m+r-2)\right]?
$$
\end{problem}
Notice that for $r=2$ the solution of Problem \ref{HF} is the Hajnal-Folkman lemma; for $r=3$ it is open, and implies the Szemer\'edi and Petruska conjecture.


\end{document}